\documentclass[12pt]{amsart}
\usepackage{amsfonts}
\usepackage{latexsym,amscd,amssymb}
\usepackage[all]{xy}
\usepackage{array}
\usepackage{hyperref}
\usepackage{graphicx}
\usepackage{epsfig}
\usepackage{subfigure}

\newtheorem*{theorem*}{Theorem}
\newtheorem*{definition*}{Definition}
\newtheorem{theorem}{Theorem}[section]
\newtheorem{proposition}[theorem]{Proposition}
\newtheorem{lemma}[theorem]{Lemma}
\newtheorem{corollary}[theorem]{Corollary}
\newtheorem{definition}[theorem]{Definition}

\newtheorem{example}[theorem]{Example}

\numberwithin{equation}{section}

\begin{document}

\title{Fibred toric varieties in toric hyperk\"{a}hler varieties}

\author{Craig Van Coevering \and Wei Zhang }

\thanks{Mathematics Classification Primary(2000): Primary 53C26, Secondary 53D20, 14L24.\\
\indent The second author is supported by Tian Yuan math Fund. and the Fundamental Research Funds for the Central Universities\\
\indent Keywords: fibred toric variety, toric hyperk\"{a}hler
variety, polytope, arrangement, chamber, core, natural morphism,
flip, Mukai flop}

\begin{abstract}

We introduce the fibred toric varieties as equivariant $\mathbb{C}P^r$ bundles over lower dimensional toric varieties. An equivalent characterization is that the natural morphisms on them degenerate to bundle projections in the context of variation of toric varieties as GIT quotients. Our main observation is that these fibred toric varieties also arise naturally in the variation of hyperk\"{a}hler varieties, namely, the fibred toric varieties are contained in the exceptional sets of the hyperk\"{a}hler natural morphisms and the Mukai flops.

\end{abstract}

\maketitle

\section{Introduction}

Toric varieties are defined by the combinatorial data of the fan (cf.
\cite{Fu93} or \cite{Od88}) as algebraic varieties, also studied by
Delzant (\cite{De88}) in symplectic quotient perspective, associated
to polytopes by moment maps, and later Guillemin (\cite{Gu94b}) proved
their canonical symplectic forms are in effect K{\"a}hler. More
over, the symplectic quotient has natural connection with the Geometric Invariant
Theory(short as GIT, cf. \cite{MFK94}).

Toric hyperk\"{a}hler varieties are quaternion analogue of toric
varieties,
which can be obtained as symplectic quotient of level set of the
holomorphic moment map. Using symplectic quotient technique, in
\cite{BD00}, Bielawski and Dancer studied their moment maps, cores,
cohomologies, etc. While Hausel and Sturmfels study the toric
hyperk{\"a}hler varieties from a more algebraic view
point(\cite{HS02}).

Choosing different level sets of moment map, conducting the symplectic quotient, we get different toric varieties
and toric hyperk\"{a}hler varieties respectively. It is natural to
ask how the varieties change as the values of moment maps change. The
symplectic quotient amounts to GIT quotient if we take the value of moment
map to be the linearization of the torus action on line bundle.
Thus the dependence of symplectic quotient to the level set is transferred to the dependence of GIT quotient to the linearization, which had been studied abstractly by Thaddeus in
\cite{Th96}, Dolgachev and Hu in \cite{DH98}. In the toric case, there is some related discussion on toric varieties in \cite{CLS11}, and Konno studied the
variation of toric hyperk\"{a}hler varieties: the natural morphisms
and Mukai flops, getting more special results(cf. \cite{Ko03},\cite{Ko08}).

%\subsection{Main results}
In this article, we give the definition of fibred toric variety, then show that it takes key position in both variations of toric varieties and toric hyperk{\"a}her varieties. The advantages in toric category are the computation can be made directly based on the definitions of GIT, and more importantly, the variation of GIT quotients can be visualized by the variation of the hyperplane arrangements which can offer us more geometric information. In detail, this article is organized as follows. In section 2, after the symplectic
quotient construction of toric variety, where the toric variety can be determined either by $\alpha$ a value of moment map or $\mathcal{A}$ a hyperplane arrangement, the fibred toric variety is defined.

\begin{definition}\label{def:eletor}
We call a $(r+s)$-dimensional toric variety $X$ fibred toric variety, if it is a $\mathbb{C}P^r$ bundle over
$s$-dimensional toric variety $X_1$, and the bundle projection map is $T^s_\mathbb{C}$ equivariant, i.e. the $T^s_\mathbb{C}$ torus action of $X_1$ lifted to $X$ as a subgroup of $T^{r+s}_\mathbb{C}$, making the following diagram commutes.
\[
\begin{CD}
X @>{\rm T^{s}_\mathbb{C} \subset T^{r+s}_\mathbb{C}}>> X \\
@VVV @VVV\\
X_1 @>{\rm T^{s}_\mathbb{C}}>> X_1\\
\end{CD}
\]

\end{definition}

In the end of section, the regularity of the toric
variety $X(\alpha)$ is precisely described by the chamber structure on $\mathfrak{m}^*_+$.

In section 3, we establish the GIT quotient construction of toric variety, which is equivalent to the former symplectic one. Then we focus on the variation of toric varieties as GIT quotient where the following theorem is proven. Here $\alpha^\pm$ are two regular value in adjacent chambers, $\alpha_1$ is a singular lying on the generic position of the wall and $\pi^\pm$ are the natural morphisms form toric varieties $X(\alpha^\pm)$ to $X(\alpha_1)$.

\begin{theorem}\label{thm:natmor}
(1) If we set $V_1 = \{[z] \in X(\alpha_1)| z \zeta = z, \ for \
\zeta\in G_1\}$, then $V_1$ is a toric variety.

\noindent (2) If we set $V^+ = (\pi^+)^{-1} (V_1)$, $V^- =
(\pi^-)^{-1} (V_1)$, then $\pi^\pm|_{V^\pm} : V^\pm \rightarrow V_1$
are fiber bundles whose fibers are biholomorphic to $\mathbb{C}P^{\#
J_1^\pm-1}$, where $\#J_1^+$ $\#J_1^-$ are the numbers of elements
in $J_1^+$ and $J_1^-$ respectively.

\noindent (3) Natural morphisms $\pi^\pm|_{X(\alpha^\pm)\backslash
V^\pm}: X(\alpha^\pm)\backslash V^\pm \rightarrow
X(\alpha_1)\backslash V_1$ are both biholomorphic maps.
\end{theorem}

In the case $\alpha_1$ lies in a wall which is the boundary of $\mathfrak{m}^*_+$, $X(\alpha^-)$ is empty, thus the natural morphism $\pi^+$ from $X(\alpha^+)$ to $X(\alpha_1)$ degenerates to a bundle projection with fiber a projective space. Hence $X(\alpha^+)$ is a fibred toric variety. And it is showed that all fibred toric variety comes in this way.

Section 4 is parallel to section 2. Most properties of toric variety
have their analogues in toric hyperk\"{a}hler case. Additionally, we
mainly review the result in \cite{Pr08} and \cite{BD00} concerning the extended core and core of toric
hyperk\"{a}hler variety, which establishes the deep connection
between toric variety and toric hyperk\"{a}hler variety.

Our discussion in section 5 on the variation of toric
hyperk\"{a}hler variety is highly influenced by Konno's works. He described toric hyperk\"{a}hler varieties as GIT quotient and studied the natural morphisms and Mukai flops of them, which take similar forms with the toric varieties.

\begin{theorem}\cite{Ko03, Ko08}\label{thm:hyp:nat}

\noindent (1) If we set $V_1= \{[z,w] \in  Y(\alpha_1, \beta) | (z,w)\zeta =
(z,w) \ \text{for} \ \zeta \in G_1\}$, then $V_1$ is a toric
hyperk\"{a}hler variety.

\noindent (2) If we set $V^\pm = (\pi^\pm)^{-1}(V_1)$, then
$\pi|_{V^\pm} : V^\pm \rightarrow V_1$ are fiber bundles whose fiber
is biholomorphic to $\mathbb{C}P^{\# J_1 -1}$. Moreover, the complex
codimension of $V_1$ and $V^\pm$ in $Y(\alpha_1, \beta)$ and
$Y(\alpha^\pm, \beta)$ are $2(\# J_1 -1)$ and $\# J_1 -1$
respectively, where $\#J_1$ is the number of elements in $J_1$.

\noindent (3) The natural morphism
$\pi^\pm|_{Y(\alpha^\pm,\beta)\backslash V^\pm} : Y(\alpha^\pm,
\beta) \backslash V^\pm \rightarrow Y(\alpha_1, \beta) \backslash
V_1$ are biholomorphic maps.

\end{theorem}

Take $\beta=0$, consider the natural morphism between $4n$ dimensional toric hyper{\"a}hler varieties $\pi^\pm: Y(\alpha^\pm,0) \rightarrow Y(\alpha_1,0)$, it can be encoded as the variation of a smooth hyperplane arrangement $\mathcal{A}$ to a non-simplicial one $\mathcal{A}_1$ with $\mathcal{S}$ a lower dimensional arrangement as singular set. We have

\begin{theorem}\label{thm:restri}

Let $Z_1$ be the extended core of $V_1$ which is the toric varieties $X(\mathcal{S}_{\tilde{\epsilon}})$ intersecting together,
Restrict $V^\pm$ these $\mathbb{C}P^r$ fiber bundles to $Z_1$, then $V^\pm|_{X(\mathcal{S}_{\tilde{\epsilon}})}$ the $\mathbb{C}P^r$ fiber bundles over each $X(\mathcal{S}_{\tilde{\epsilon}})$ are all fibred toric varieties of complex dimension $n$.

\end{theorem}

\noindent \textbf{Acknowledgement:} Both authors want to thank Prof.
Bin Xu, Prof. Bailin Song, Dr. Yalong Shi and Dr. Yihuang Shen for valuable
conversations and the second author want to thank his supervisor
Prof. Yuxin Dong for constant encouragement.

\section{Toric variety}
Various descriptions of toric variety have their own advantages. We
first consider the symplectic quotient, then shift to GIT quotient
for the study of variation in next section.
\subsection{Symplectic quotient(K{\"a}hler quotient)}
The real torus $T^d=\{(\zeta_1, \zeta_2, \cdots, \zeta_d) \in
\mathbb{C}^d, |\zeta_i|=1\}$ acts on $\mathbb{C}^d$ freely. Denote
$M$ the $m$-dimensional connected subtorus of $T^d$ whose Lie
algebra $\mathfrak{m} \subset \mathfrak{t}^d$ is generated by
integer vectors(which is always taken to be primitive), then we have
the following exact sequences
\begin{equation*}
0\rightarrow \mathfrak{m}\xrightarrow{\iota}
\mathfrak{t}^d\xrightarrow{\pi} \mathfrak{n}\rightarrow 0,
\end{equation*}
\begin{equation*}\label{eq:exact:dual}
0\leftarrow {\mathfrak{m}}^*\xleftarrow{\iota^*}
(\mathfrak{t}^d)^*\xleftarrow{\pi^*} {\mathfrak{n}}^*\leftarrow 0,
\end{equation*}
where $\mathfrak{n}=\mathfrak{t}^d/\mathfrak{m}$ is the Lie algebra
of the $n$-dimensional quotient torus $N=T^d/M$ and $m+n=d$. For
simplicity, we omit the superscript $d$ over $\mathfrak{t}$ from now
on.

Let $\{e_i\}_{i=1}^d$ be the standard basis of $\mathfrak{t}$, then
$\pi(e_i)=u_i$ are also primitive. Denote $\{e^*_i\}_{i=1}^d$ the
dual basis of $\mathfrak{t}^*$ and $\{\theta_i\}_{i=1}^m$ some basis
span $\mathfrak{m}$. The action of $M$ on $\mathbb{C}^d$ admits a
moment map
\begin{equation*}
\mu(z)=\frac{1}{2}\sum_{i=1}^d|z_i|^2\iota^* e^*_i.
\end{equation*}
For any $\alpha \in \mathfrak{m}^*$, the symplectic quotient
$\mu^{-1}(\alpha)/M$ is a toric variety, denoted as $X(\alpha)$,
inheriting K{\"a}hler metric from $\mathbb{C}^d$ on
it's smooth part(cf. \cite{Gu94a}).

The quotient torus $N$ has a residue circle action on $X(\alpha)$
and gives rise to a moment map to $\mathfrak{n}^*$,
\begin{equation*}
\bar{\mu}([z])=\frac{1}{2}\sum_{i=1}^d|z_i|^2 e^*_i.
\end{equation*}
The image of this map is a convex polytope $\Delta$ called the
Delzant polytope of $X(\alpha)$(cf. \cite{De88}). Conversely, any smooth compact toric variety $X$ of complex
dimension $n$, with a K{\"a}hler metric invariant under some torus $N$
comes from Delzant's construction. Unfortunately, this polytope does
not recover all the data of the quotient construction, and the worse
is that it does not cooperate well with the toric hyperk{\"a}hler
theory. We use the notion of hyperplane arrangement with
orientation(cf. \cite{Pr04}) to replace polytope. In detail,
consider a set of rational oriented hyperplanes
$\mathcal{A}=\{(H_i,u_i)\}_{i=1}^d$,
\begin{equation*}
H_i=\{x \in \mathfrak{n}^*|\langle u_i,x\rangle+\lambda_i=0\},
\end{equation*}
where $H_i$ is the hyperplane and $u_i$ is fixed primitive vector in
$\mathfrak{n}_\mathbb{Z}$ specifying the orientation, called the
normal of $H_i$. We define several subspaces related to these
oriented hyperplanes,
\begin{equation*}
H^{\geq 0}_i=\{x \in \mathfrak{n}^*|\langle u_i,x \rangle+\lambda_i
\geq 0\},
\end{equation*}
\begin{equation*}
H^{\leq 0}_i=\{x \in \mathfrak{n}^*|\langle u_i,x \rangle+\lambda_i
\leq 0\}.
\end{equation*}
A ploytope is naturally associated to this arrangement,
\begin{equation*}
\Delta=\bigcap^d_{i=1} H^{\geq 0}_i,
\end{equation*}
which could be empty or unbounded.

The arrangement $\mathcal{A}$ will decide a toric variety the
same as $\Delta$ does. Since $u_i$ define a map $\pi: \mathfrak{t}
\rightarrow \mathfrak{n}$, where $\mathrm{Ker} \pi=\mathfrak{m}$,
let $M$ be the Lie group corresponding to $\mathfrak{m}$ and set
$\alpha=\sum \lambda_i \iota^* e_i^*$, then we call
$\mu^{-1}(\alpha)/M$ the toric variety corresponding to
$\mathcal{A}$ and $\lambda=(\lambda_1, \cdots, \lambda_d)$ a lift of
$\alpha$. For fixed normal vectors, the hyperplane arrangements
corresponding to two different lifts of same moment map value
$\alpha$ only differ by a parallel transport, thus produce same
toric variety(cf. \cite{Pr04}). So we can abuse the notations
$X(\alpha)$ and $X(\mathcal{A})$. Moreover, denote $\Theta$ the set of maps
form $\{1, \dots, d\}$ to $\{-1,1\}$. For $\epsilon \in \Theta$, let
$\mathcal{A}_\epsilon$ be the arrangement changing the normal of
$H_i$ if $\epsilon(i)=-1$, and when $\epsilon(i)=1$ for all $i$, we
abbreviate the subscript $\epsilon$ for simplicity. Notice that the
toric variety $X(\mathcal{A}_\epsilon)$ for various $\epsilon$ could
be totally different.

\begin{example}[see \cite{BD00}]\label{ex:1dim:arr}
We take $u_1=-f_1$, $u_2=u_3=f_1$ in $\mathfrak{n}^1$ where $f_1$ is
the standard basis, and $\lambda_1=1$, $\lambda_2=-\frac{1}{2}$,
$\lambda_3=0$. Hence $\mathfrak{m}$ is spanned by $\iota
\theta_1=e_1+e_2$, $\iota \theta_2=e_1+e_3$, for short, $(1,1,0)$
and $(1,0,1)$. The toric variety is $\mathbb{C}P^1$. See Figure
\ref{fig:1dim}.
\end{example}

\begin{figure}[h]
\centering \subfigure[the fan, where $u_2$, $u_3$ superposition]{
\label{fig:1dim:fan} %% label for first subfigure
\includegraphics[width=2.2in]{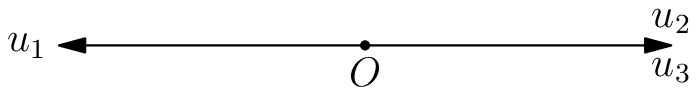}}
\hspace{0.3in} \subfigure[the hyperplane arrangement and polytope
which is represented by the bold line]{
\label{fig:1dim:arr} %% label for second subfigure
\includegraphics[width=2.2in]{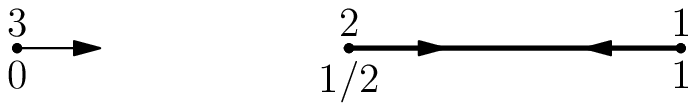}}
\caption{the fan and hyperplane arrangement in Example
\ref{ex:1dim:arr}}
\label{fig:1dim} %% label for entire figure
\end{figure}

\begin{example}[see \cite{BD00} or \cite{Pr04}]\label{ex:2dim:arr}
Let $n=2$, $u_1=f_1$, $u_2=f_2$, $u_3=-f_1-f_2$, $u_4=-f_2$, and
$\lambda_1=\lambda_2=\lambda_3=\lambda_4=1$. The toric variety is
Hirzebruch surface $\mathbb{P}(\mathcal{O}\oplus\mathcal{O})$. See
Figure \ref{fig:2dim}.
\end{example}

\begin{figure}[h]
\centering \subfigure[the fan]{
\label{fig:2dim:fan} %% label for first subfigure
\includegraphics[width=2.2in]{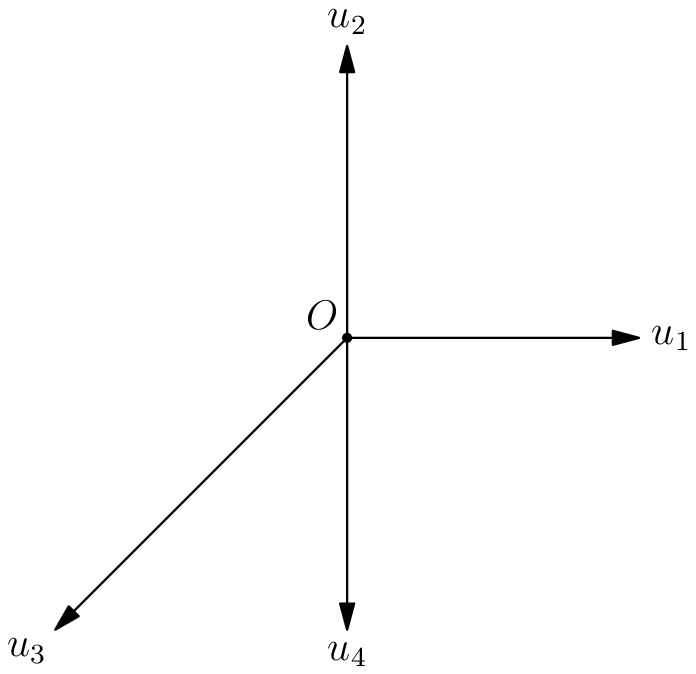}}
\hspace{0.3in} \subfigure[the hyperplane arrangement and polytope
which is represented by the shade area]{
\label{fig:2dim:arr} %% label for second subfigure
\includegraphics[width=2.2in]{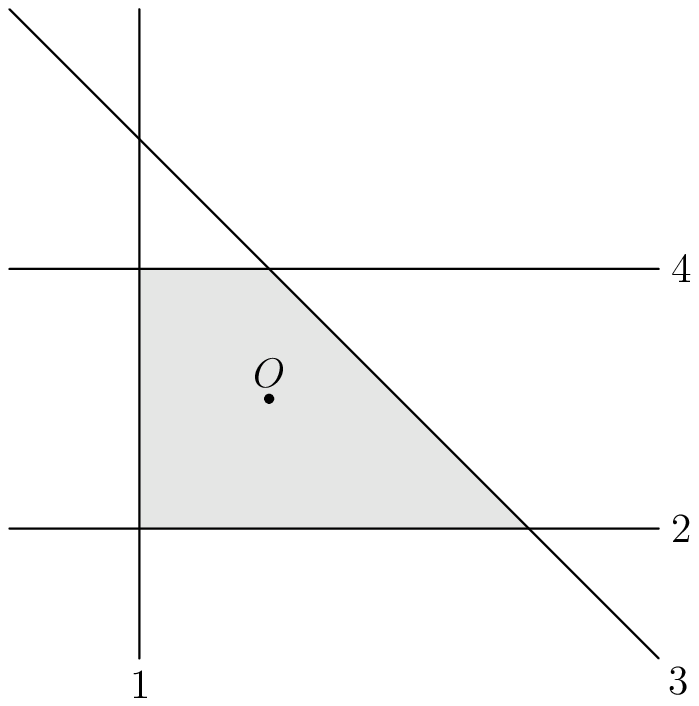}}
\caption{the fan and hyperplane arrangement in Example
\ref{ex:2dim:arr}}
\label{fig:2dim} %% label for entire figure
\end{figure}

\subsection{Definition of fibred toric variety}

Now we state the definition of fibred toric variety.

\begin{definition*}
We call a $(r+s)$-dimensional toric variety $X$ fibred toric variety, if it is a $\mathbb{C}P^r$ bundle over
$s$-dimensional toric variety $X_1$, and the bundle projection map is $T^s_\mathbb{C}$ equivariant, i.e. the $T^s_\mathbb{C}$ torus action of $X_1$ lifted to $X$ as a subgroup of $T^{r+s}_\mathbb{C}$, making the following diagram commutes.

\[
\begin{CD}
X @>{\rm T^{s}_\mathbb{C} \subset T^{r+s}_\mathbb{C}}>> X \\
@VVV @VVV\\
X_1 @>{\rm T^{s}_\mathbb{C}}>> X_1\\
\end{CD}
\]

\end{definition*}

The first nontrivial fibred toric variety is the Hirzebruch surface in Example \ref{ex:2dim:arr}, a $\mathbb{C}P^1$ bundle over $\mathbb{C}P^1$. And we will see this kind of toric varieties is very important in the variation of toric varieties and toric hyperk{\"a}hler varieties.

\subsection{Regularity} If the rational vectors $u_i$
is regular, i.e. every collection of $n$ linearly independent
vectors $\{u_{i_1},\cdots, u_{i_n}\}$ span $\mathfrak{n}_\mathbb{Z}$
as a $\mathbb{Z}$-basis, then $\mathcal{A}$ is called regular. The
arrangement $\mathcal{A}$ is called simplicial if every subset of
$k$ hyperplanes with nonempty intersection intersects in codimension
$k$. For a non-simplicial arrangement, all the points in the
intersection of $k$ hyperplanes whose codimension is lower than $k$ are called singular set.
$\mathcal{A}$ is smooth if it is both regular and simplicial.

\begin{theorem}[\cite{BD00},\cite{Pr04}]
$X(\mathcal{A})$ is an orbifold, if and only if $\mathcal{A}$ is
simplicial, and $X(\mathcal{A})$ is smooth if and only if
$\mathcal{A}$ is smooth. When $\mathcal{A}$ is regular but non
simplicial, it may attain Ablean quotient singularity.
\end{theorem}
From now on, $\mathcal{A}$ is always assumed regular to exclude the
general orbifold singularities. The smoothness of hyperplane
arrangement is closed related to the the moment map's regularity.
Let $\mathfrak{m}^*_{+}$ be the positive cone spanned by
$\mathbb{R}_{\geq 0} \iota^* e_i^*$. Denote the isotropy subgroup of
$M$ at $z \in \mathbb{C}^d$ by $M_z$, set $\Lambda = \{M_z| z \in
\mathbb{C}^d\}$ and
\begin{equation*}
\Lambda(k) = \{G\in \Lambda | \mathrm{dim} G = k\} \ \text{for}
\quad k = 0,\cdots, m.
\end{equation*}
Focusing on the set of one dimensional isotropy groups $\Lambda(1) =
\{G_s\}_{s=1}^l$, we call the subspace in $\mathfrak{m}^*$ of
codimension one
\begin{equation*}
W_s=\{v^* \in \mathfrak{m}_+^*|\langle v^*, \mathrm{Lie}
G_s\rangle=0\} \subset \mathfrak{m}_+^*
\end{equation*}
a wall. Notice that the wall $W_s$ is spanned by $\{\mathbb{R}^+
\iota^* e_i^* |\langle\iota^* e_i^*, \mathrm{Lie} G_s \rangle= 0\}$.

\begin{proposition}
$$\mathfrak{m}^*_{+reg}=\mathfrak{m}^*_{+} \backslash \bigcup^l_{s=1} W_s$$
\end{proposition}

\begin{proof}
Let $f: \mathbb{C}\rightarrow \mathbb{R}$ be a map defined by
$f(a)=|a|^2$. We can easily observe that $a \in \mathbb{C}$ is a
regular point of $f$ if and only if $a \neq 0$. Since
$(\mathrm{d}\mu)(z) = \sum_{i=1}^d (\mathrm{d}f)_{z_i} \otimes
\iota^* e_i^* \in \mathfrak{m}^*$, $(z) \in \mathbb{C}^d$ is a
critical point of $\mu$ if and only if span $\{\iota^* e_i^*|z_i\neq
0\} \subsetneqq \mathfrak{m}^*$. This implies the proposition.
\end{proof}

There is an one to one correspondence between $\alpha$ and
$\mathcal{A}$, if the fan $\{u_i\}$ is regular, then the smoothness
of $\mathcal{A}$ and the regularity of moment map at $\alpha$
coincide. To see this, for fixed $\alpha \in \mathfrak{m}^*$, $\lambda$ is one of its lift, consider the equation
\begin{equation*}
\alpha=\sum_{i=1}^d x_i \iota^* e_i^*.
\end{equation*}
Denoting the solution space as $\mathfrak{N}_\alpha$, a $n$-plane in $\mathbb{R}^d$, we have

\begin{proposition}\label{pro:arrcut}
Regarding the point $\lambda$ in $\mathfrak{N}_\alpha$ as the origin
, projecting of $\mathfrak{N}_\alpha$ onto some standard
$\mathbb{R}^n$, then we can identify
$\pi_{\mathbb{R}^n}(\mathfrak{N}_\alpha)$ with $\mathfrak{n}^*$, and
the hyperplane arrangement $H_i$ is defined by
$\pi_{\mathbb{R}^n}(\mathfrak{N}_\alpha \bigcap \{x_i=0\})$, where
$\{x_i=0\}$ is the coordinate hyperplane in $\mathbb{R}^d$.
\end{proposition}

The proof can be found in \cite{vCZ11}. Consequentially, $\mathcal{A}$ is non-simplicial means that there is
a point $\lambda^\prime$ in $\mathfrak{N}_\alpha$ more than $n+1$
hyperplanes intersect in, i.e. more than $n+1$ coordinates
$\lambda_i^\prime$ are zero. For $\lambda^\prime$ is also a new lift of
$\alpha$, and $\alpha=\lambda_i^\prime \iota^* e_i^*$ at most
involve $m-1$ terms, thus $\alpha$ must lie in a wall and vice
versa.

Reader should notice that, even if $\alpha$ is singular or
equivalently $\mathcal{A}$ is non-simplicial, the toric variety
still has possibility to be smooth. For instance, $\lambda=(1,0,0)$,
then $\alpha=\theta_1^*$ in Example \ref{ex:1dim:arr}, gives rise to
$\mathbb{C}P^1$; or $\lambda=(1,1,1,2)$, then
$\alpha=3(\theta_1^*+\theta_2^*)$ in Example \ref{ex:2dim:arr},
gives rise to $\mathbb{C}P^2$.

The connected components of $\mathfrak{m}^*_{+reg}$ are called
chambers. Within a chamber, the toric varieties are all
biholomorphic. Therefore the only interesting variations are
$\alpha$ moving into the wall or across the wall.

\begin{example}\label{ex:fanchanged}
For the fan in Example \ref{ex:2dim:arr}, we have $\iota^*
e_1^*=\iota^* e_3^*=\theta_1^*$, $\iota^*
e_2^*=\theta_1^*+\theta_2^*$, $\iota^* e_4^*=\theta_2^*$. It's
chamber is drawn in Figure \ref{fig:tor:chamber:a}. If we change the
fan to: $u_1=f_1$, $u_2=u_4=-f_2$, $u_3=-f_1-f_2$, i.e. change the
direction of $u_2$, then we will have $\iota^* e_1^*=\iota^*
e_3^*=\theta_1^*$, $\iota^* e_2^*=-\theta_1^*-\theta_2^*$, $\iota^*
e_4^*=\theta_2^*$, as Figure \ref{fig:tor:chamber:b}.
\end{example}

\begin{figure}[h]
\centering \subfigure[the chamber of Example \ref{ex:2dim:arr}]{
\label{fig:tor:chamber:a} %% label for first subfigure
\includegraphics[width=2.0in]{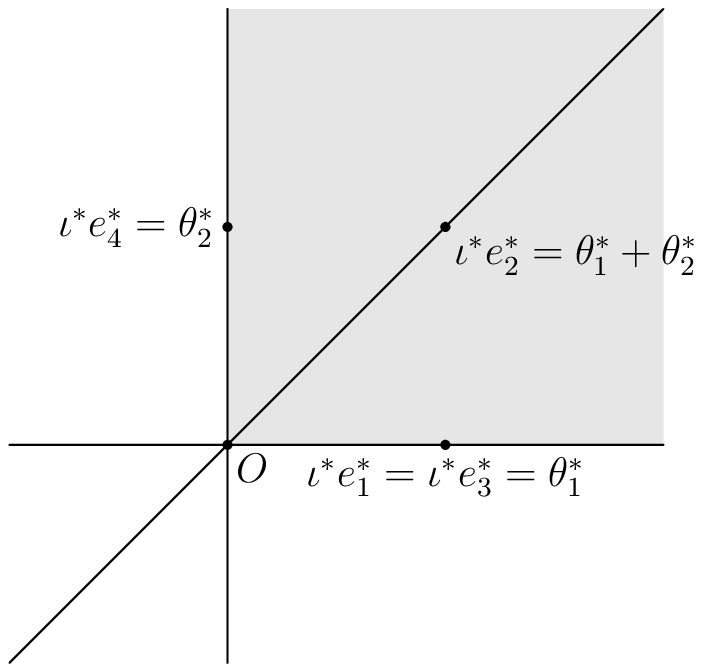}}
\hspace{0.3in} \subfigure[the chamber of same arrangement but
different orientation]{
\label{fig:tor:chamber:b} %% label for second subfigure
\includegraphics[width=2.4in]{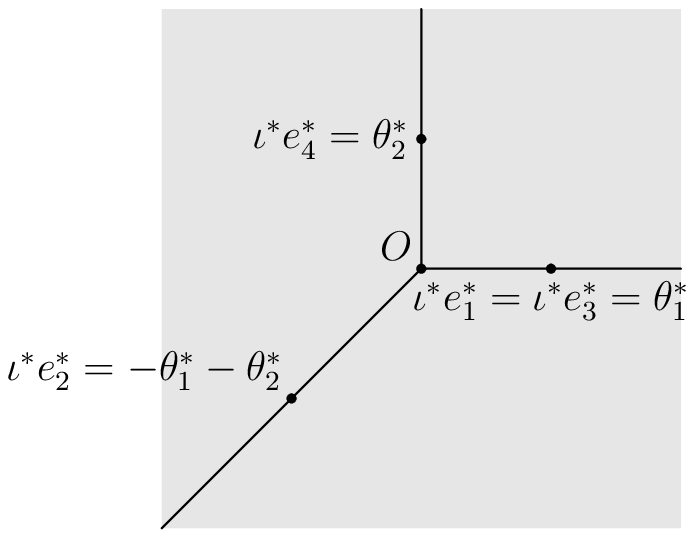}}
\caption{two examples of chamber}
\label{fig:tor:chamber} %% label for entire figure
\end{figure}

\section{Variation of toric variety}
Since symplectic quotient is not suitable for the study of
variation, a more intrinsic way to describe toric variety is
necessary.

\subsection{GIT quotient}

Let us consider the GIT quotient of $\mathbb{C}^d$ by $M_\mathbb{C}$
with respect to the linearization induced by $\alpha \in
\mathfrak{m}^*_\mathbb{Z}$. The element $\alpha$ induces the
character $ \chi_\alpha: M_\mathbb{C} \rightarrow \mathbb{C}^\times$
by $\chi_\alpha(\mathrm{Exp} (\rho \theta))=e^{\langle \alpha,
\theta \rangle\rho}$, where $M_\mathbb{C}$ is the complexfication of
$M$ and $\theta \in \mathfrak{m}$, $\rho \in \mathbb{C}$. Let $L^{\otimes m} = \mathbb{C}^d
\times \mathbb{C}$ be the trivial holomorphic line bundle on which
$M_\mathbb{C}$ acts as
\begin{equation*}
((z), v)_m \zeta = ((z) \zeta, v \chi_\alpha(\zeta)^m)_m.
\end{equation*}
A point $(z)$ is $\alpha$-semi-stable if and only if there exists
$m\in \mathbb{Z}_{>0}$ and a polynomial $f(p)$, where $p \in
\mathbb{C}^d$, such that $f(p)$ viewed as a section of $L^{\otimes
m}$ is invariant under the action, that is $f((p)\zeta)=
f(p)\chi_\alpha(\zeta)^m$ for any $\zeta \in M_\mathbb{C}$, and
additionally $f(z) \neq 0$. We denote the set of
$\alpha$-semi-stable points in $\mathbb{C}^d$ by
$(\mathbb{C}^d)^{\alpha-ss}$, then there is a categorical quotient
$\phi: (\mathbb{C}^d)^{\alpha-ss} \rightarrow
(\mathbb{C}^d)^{\alpha-ss}//M_{\mathbb{C}}$, where
$(\mathbb{C}^d)^{\alpha-ss}//M_{\mathbb{C}}$ is the GIT quotient of
$\mathbb{C}^d$ by $M_\mathbb{C}$ respect to $\alpha$(cf.
\cite{MFK94}, and the readers are highly recommended to consult the
lecture notes \cite{Do03} or \cite{Th06} if they prefer the variety
rather than the abstract scheme setup). Sometimes
$\mathbb{C}^d//_\alpha M_{\mathbb{C}}$ stands for this GIT quotient.
We cite several basic properties of GIT quotient without proof.
\begin{lemma}
For any point $p \in (\mathbb{C}^d)^{\alpha-ss}//M_{\mathbb{C}}$,
the fiber $\phi^{-1}(p)$ consists of finitely many
$M_{\mathbb{C}}$-orbits. Moreover, each fiber contains the unique
closed $M_{\mathbb{C}}$-orbits in $(\mathbb{C}^d)^{\alpha-ss}$. Thus
the categorical quotient
$(\mathbb{C}^d)^{\alpha-ss}//M_{\mathbb{C}}$ can be identified with
the set of closed $M_{\mathbb{C}}$-orbits in
$(\mathbb{C}^d)^{\alpha-ss}$.
\end{lemma}

Unfortunately, the definition of stability respect to linearization
is only effective when $\alpha \in \mathfrak{m}^*_\mathbb{Z}$, i.e.
only corresponds to the algebraic toric variety with line bundle
described by Newton ploytope with integer vertices. Following Konno's
method in hyperk\"{a}hler case, it can be generalized to any complex
manifold.

\begin{definition}\label{def:sta:num}
Suppose that $\alpha \in \mathfrak{m}^*$,

\noindent (1) A point $z \in \mathbb{C}^d$ is $\alpha$-semi-stable
if and only if
\begin{equation}\label{eq:sta:num}
\alpha \in \sum^d_{i=1} \mathbb{R}_{\geq 0}|z_i|^2 \iota^* e_i^*.
\end{equation}

\noindent (2) Suppose $z \in (\mathbb{C}^d)^{\alpha-ss}$. Then the
$M_\mathbb{C}$-orbit through $z$ is closed in
$(\mathbb{C}^d)^{\alpha-ss}$ if and only if
\begin{equation}\label{eq:clo:num}
\alpha \in \sum^d_{i=1} \mathbb{R}_{> 0}|z_i|^2 \iota^* e_i^*.
\end{equation}
\end{definition}
This definition of stability coincides the original GIT one when
$\alpha \in \mathfrak{m}^*_\mathbb{Z}$. For convenience we give the proof of (1), and readers can find the
essential proof of (2) in \cite{Ko08}. Suppose $(z) \in
(\mathbb{C}^d)^{\alpha-ss}$. Then there exists $m \in
\mathbb{Z}_{>0}$ and a polynomial $f(p_1, \dots, p_d)$ such that
$f((p)\zeta) = f(p) \chi_\alpha(\zeta)^m$ and $f(z)\neq 0$. So we
can select out a monomial $f_0(p) = \prod_{i=1}^d p_i^{a_i}$, where
$a_i \in \mathbb{Z}_{>0}$, such that
$f_0((p)\zeta)=f_0(p)\chi_\alpha(\zeta)^m$ and $f_0(z)\neq0$. The
second condition implies that $a_i = 0 \ \text{if} \ z_i = 0$.
Moreover, the first condition implies $m \alpha = \sum_{i=1}^d a_i
\iota^* e_i^*$. To see this, let $\theta_k$ be the standard basis of
$\mathfrak{m}$ and $\rho \in \mathbb{C}^\times$, we have
$\chi(\mathrm{Exp}(\rho \theta_k))=e^{\rho\langle \alpha, \theta_k
\rangle}$ and $(p)\mathrm{Exp}(\rho \theta_k)=(p_i
e^{\rho\langle\iota\theta_k,e_i^*\rangle})=(p_i e^{\rho\langle
\iota^* e_i^*, \theta_k\rangle})$. Thus we proved Equation
(\ref{eq:sta:num}).

\begin{proposition}

\noindent (1)If we fix $\alpha \in \mathfrak{m}^*$, then the natural
map $\sigma: X(\alpha) \rightarrow
(\mathbb{C}^d)^{\alpha-ss}//M_{\mathbb{C}}$ is a homeomorphism(if
$\alpha \notin \mathfrak{m}_+^*$, both sets are empty).

\noindent (2)If $\alpha \in \mathfrak{m}_{+reg}^*$, then every
$M_\mathbb{C}$-orbit is closed in $(\mathbb{C}^d)^{\alpha-ss}$. So
the categorical quotient
$(\mathbb{C}^d)^{\alpha-ss}//M_{\mathbb{C}}$ is a geometric quotient
$(\mathbb{C}^d)^{\alpha-ss}/M_{\mathbb{C}}$.
\end{proposition}

Readers could consult \cite{Ko08} for the proof. Thus the symplectic
quotient $X(\alpha)$ can be identified with the GIT quotient
$(\mathbb{C}^d)^{\alpha-ss}//M_{\mathbb{C}}$ in both algebraic and
holomorphic case. This principle was established in \cite{KN78},
\cite{MFK94} in the algebraic case. The general holomorphic version
is proved in \cite{Na99}. Variation of a toric variety with respect to fixed fan, means
changing $\alpha$ the value of moment map in $\mathfrak{m}^*$(more
precisely $\mathfrak{m}_+^*$), equals to altering the linearization
of the GIT quotient. The variation of abstract GIT quotient had been studied thoroughly in \cite{Th96}
and \cite{DH98}, but the toric variety case has its
independent interest. This is because how $X(\alpha)$ varies can be read off from the variation of $\mathcal{A}$ directly. We will reprove
some of their results in toric variety case, leading to more special
consequence.

\subsection{Natural morphism and flip}

The phenomena of $\alpha$ moving from the interior of chamber into
the wall is described by the so called natural morphism.
Suppose $\alpha_1$ is in a generic\footnote{Does not lie in the
intersection with any other wall.} position of the wall $W_1$ and
$\alpha^+$, $\alpha^-$ lie in the chamber $\mathcal{C}^+$ and
$\mathcal{C}^-$, s.t. $W_1 \subset \overline{\mathcal{C}^\pm}$. By
the stability condition Equation (\ref{eq:sta:num}),
$(\mathbb{C}^d)^{\alpha^\pm-ss} \subset
(\mathbb{C}^d)^{\alpha_1-ss}$, inducing the natural morphisms
between GIT quotients $\pi^\pm : X(\alpha^\pm) \rightarrow
X(\alpha_1)$. The toric variety $X(\alpha_1)$ may have singularity and these natural
morphisms are kinds of ``desingularization". To see this, let $G_1$
be the 1-dimensional isotropy subgroup orthogonal to the wall
$W_1$, and $\theta_1$ is a non-zero element in $\mathrm{Lie} G_1$
s.t. $\langle\theta_1, \alpha_1\rangle>0$, set $J_1 = \{i\in
\{1,\cdots,d\}| \langle \iota^*e_i^*,\theta_1\rangle \neq 0\}$, then
$J_1$ divides into $J_1^+ = \{i \in J_1\}|
\langle\iota^*e_i^*,\theta_1\rangle > 0\}$ and $J_1^+ = \{i \in
J_1\}| \langle\iota^*e_i^*,\theta_1\rangle < 0\}$. Following Konno's
methods in hyperk\"{a}hler case, we can prove Theorem \ref{thm:natmor}.

\noindent \textbf{The proof of Theorem \ref{thm:natmor}:} (1) Note
that $W_1$ can be identified with the dual space of the Lie algebra
$\mathfrak{m}_1$ of the quotient torus $M_1= M/G_1$. So $\alpha$ can
be considered as a regular element of $\mathfrak{m}_1$. Then $V_1$
is a K{\"a}hler quotient of $\mathbb{C}^{\{1,...,d\}\backslash J_1}=
\{(z) \in \mathbb{C}^d | z_i= 0 \ if \ i \in J_1\}$ by $M_1$.

\noindent (2) Choosing $\theta_1 \in \mathrm{Lie} G_1$, by Equation
(\ref{eq:sta:num}) we know that $(\mathbb{C}^d)^{\alpha^+-ss}$ is
exactly the points in $(\mathbb{C}^d)^{\alpha_1-ss}$ satisfying
\begin{equation}\label{eq:belongsta}
\text{there exists} \ i \in J_1^+ \ \text{such that} \ z_i \neq 0.
\end{equation}
For $\alpha^+$ is a regular value in $\mathfrak{m}^*$ and $\alpha_1$
a regular value in $\mathfrak{m}_1^*$, every orbit is closed
respectively. Thus $V^+$ and $V_1$ are both geometric quotient, and
$\pi^+|_{V^+} : V^+ \rightarrow V_1$ can be interpreted as replacing
$z_i$ with 0 in $[z]$ for any $i \in J_1^+$. Notice that if $(z) \in
(\mathbb{C}^d)^{\alpha^+-ss}$, then
\begin{equation}\label{eq:belongsin}
[z] \in V^+ \  \text{is equivalent to} \ z_i = 0 \ \text{for all} \
i \in J^-_1.
\end{equation}
Thus the fiber of $\pi^+|_{V^+} : V^+ \rightarrow V_1$ is
biholomorphic to $(\mathbb{C}^{\#J_1^+} \backslash
\{0\})/G_{1\mathbb{C}}$, i.e. $\mathbb{C}P^{\#J_1^+-1}$. The case of
$\pi^-$ is tautological.

\noindent (3) If $(z)\in (\mathbb{C}^d)^{\alpha_1-ss}$, by Equation
(\ref{eq:belongsta}) and (\ref{eq:belongsin}), then $[z] \in
X(\alpha^+)\backslash V^+$ equals to
\begin{equation*}
\text{there exists} \ i \in J_1^+ \ \text{such that} \ z_i \neq 0 \
\text{and} \ j \in J_1^- \ \text{such that} \ z_j \neq 0.
\end{equation*}
This means nothing but
\begin{equation*}
X(\alpha^+)\backslash
V^+=((\mathbb{C}^d)^{\alpha^+-ss}\cap(\mathbb{C}^d)^{\alpha^--ss})/M_\mathbb{C},
\end{equation*}
and similarly
\begin{equation*}
X(\alpha^-)\backslash
V^-=((\mathbb{C}^d)^{\alpha^--ss}\cap(\mathbb{C}^d)^{\alpha^+-ss})/M_\mathbb{C}.
\end{equation*}
On the other hand, by Equation (\ref{eq:clo:num}) in Definition
\ref{def:sta:num}, $(z) \in
(\mathbb{C}^d)^{\alpha^+-ss}\cap(\mathbb{C}^d)^{\alpha^--ss}$
implies the orbit $(z)M_\mathbb{C}$ is closed in
$(\mathbb{C}^d)^{\alpha_1-ss}$, finishing the proof.   \qed

Recall there are two different kinds of walls in $\mathfrak{m}_+^*$:
one is the boundary of $\mathfrak{m}^*_{+}$(may be boundaryless, for
example Figure \ref{fig:tor:chamber:b}), called boundary wall; the
other lies in interior of $\mathfrak{m}^*_{+}$, called interior
wall.

We focus on the boundary wall. Since $J_1^-$ is empty,
$(\mathbb{C}^d)^{\alpha^--ss}$ is empty and $X(\alpha^-)\backslash
V^-$ is empty either. Thus by (3) in Theorem \ref{thm:natmor},
$\pi^+$ degenerates to a bundle projection from $X(\alpha^+)$ to
$X(\alpha_1)$ with fiber $\mathbb{C}P^{\#J_1-1}$. In effect, we have
another way to see it directly. For $\alpha_1$ is on the boundary,
by the closeness criterion Equation (\ref{eq:clo:num}), the
coefficients before $\iota^* e_i^*, i \in J_1$ must all be zero,
i.e. the only closed orbit in $(\mathbb{C}^d)^{\alpha_1-ss}$ is the
orbit through the point $\{z|z_i=0, \ for \ i \in J_1\}$. These
points have a common isotropy group $G_1$, hence $X(\alpha_1)$ can
be viewed as a lower dimensional smooth toric variety corresponding
to the group action of $M_1=M/G_1$ on
$\mathbb{C}^{\{1,\dots,d\}\backslash J_1}$, and $\alpha_1 \in
(\mathfrak{m}^*_1)_{reg}$. Thus we have $V_1=X(\alpha_1)$ and
$V^+=X(\alpha^+)$, and $\pi^+: X(\alpha^+) \rightarrow X(\alpha_1)$
is a fiber bundle whose fiber is biholomorphic to $\mathbb{C}P^{\#
J_1-1}$(the general case refers to the material in \cite{Th96},
Corollary 1.13 and below).

In the interior wall case, since we know little about how $J_1$
divides into $J_1^+$ and $J_1^-$, the natural morphism related is
much more complicated. This kind of natural morphism can not
degenerate to bundle projection, and roughly speaking, is the
generalization of blowup.

We are ready to investigate the variation of $\alpha$
cross the interior wall. The cross wall phenomena can be recovered
by the two adjacent natural morphisms. For both natural morphisms
$\pi^+$, $\pi^-$ can not degenerate, the exceptional sets $V^+$ and
$V^-$ of $V_1$ are lower dimensional. Thus $\pi^+ \circ
(\pi^-)^{-1}: X(\alpha^+)\backslash V^+ \rightarrow
X(\alpha^-)\backslash V^-$ is a biholomorphic map, and a birational
map between $X(\alpha^+)$ and $X(\alpha^-)$, which is called flip by
Thaddeus(cf. \cite{Th96}).

\subsection{Relation with fibred toric variety}

Based on above discussion, we know that for a given toric variety $X(\alpha)$, if $\alpha$ lies in the chamber next to the boundary wall, then $X(\alpha)$ is a fibred toric variety. In fact, we will show that all fibred toric varieties come from this kind construction.

We'd better give some configuration of the arrangement and the
polytope of the variation of the fibred toric variety. By
Proposition \ref{pro:arrcut}, up to biholomorphic equivalent class,
the variation of $\alpha$ into the wall or cross the wall is
equivalent to moving some hyperplane $H_i$, $i \in J_1$ in
$\mathcal{A}$ into or cross singular set of $\mathcal{A}_1$. Thus
fibred toric variety is characterized as the hyperplane $H_i$ will
never cross any other vertex of its polytope $E$ before the volume
of this polytope tends to zero, i.e. $E$ degenerates to a lower
dimensional polytope. Based on this observation, we can give a
configuration of the fibred polytope.

\begin{proposition}
The polytope $E$ defines a fibred toric variety has form of a
product: $E=\Delta(r) \times F$ where $\Delta(r)$ is the standard
polytope defining $\mathbb{C}P^r$, and $F$ is a $s$-dimensional
polytope.
\end{proposition}

\begin{proof}
Let $H$ be the moving hyperplane, move $H$ to a non-simplicial
arrangement, we denote the singular set as $F$. Restricting $F$ to
$E$, it is a face of the polytope $E$, and itself a lower
dimensional smooth polytope, still denoted as $F$. At the same time,
the restriction of $H$ to $E$ is also a smooth polytope of dimension
$n-1$, still denoted as $H$.

We can first assume that $E$ is bounded. Suppose $H$ has $q$ vertices, and $F$ has $p$ vertices $\{o_i, i=1,
\dots, p\}$. For $E$ is smooth, each vertex has $n$ edges. The proof divides into three steps as
follows.

Step1, notice the following two facts: one is the each vertex $o_i$
in $F$ has $r$ edges come from $F$. Another is $H$ can smoothly
shrink to $F$ without passing any other vertex, so all other $s$
edges of $o_i$ come from the vertices belong to $H$. This means $s$
vertices of $H$ will map to one vertex in $F$, i.e. $q=p \cdot s$.
Thus the vertices of $H$ can be labeled as $\{a_i^j\}, j=1,\dots,s$,
divide into $p$ groups.

Step2, consider the edges between different groups, we claim that
there are at most $r$ edges shed from one vertex in $H$ to the
vertices in other groups. This is because that by the parallel
moving, we know all the edges shed from one vertex $a_1^j$ in $H$
come from parallel moving the edges shed from $o_1$ in $F$, while
there are only $r$ edges shed from $o_1$.

So there are at most $\frac{q \cdot r}{2}$ edges between different
groups in $H$. Hence the sum of edges inside each groups are at
least:
\begin{equation}\label{eq:edgenum}
\frac{q(n-1)}{2}-\frac{q \cdot r}{2}=\frac{q(s-1)}{2}=\frac{p \cdot
s(s-1)}{2},
\end{equation}
where $\frac{q(n-1)}{2}$ is number of total edges in the polytope
$H$.

Step3, for in each individual group, there is at most
$\frac{s(s-1)}{2}$ edges between $s$ points, by Equation
(\ref{eq:edgenum}), each vertex must have exactly $r$ edges joining
other groups, and $s-1$ edges joining other $s-1$ point in its own
group constituting a $\Delta(s-1)$.

In conclusion $H$ can be written as $\Delta(s-1)\times F$, then $E$
is of the form $\Delta(s)\times F$.

If $E$ is unbounded, firstly, we apply above argument to its bounded
faces, combining the fact each radial in $H$ comes form the radial
in $F$ by parallel moving, accomplish the proof.
\end{proof}

For a fibred toric variety $X$, we endow a fixed fiber
$\mathbb{C}P^r$ the Fubini-Study metric, pull back the Fubini-Study
metric to all the fibers by the $T^s_\mathbb{C}$ action. Then we
give $X_1$ the $T^s$ invariant K{\"a}hler metric, and $X$ the
horizontal part metric which makes the projection a Riemann
submersion. For the projection is $T^s_\mathbb{C}$ equivariant, the
K{\"a}hler metric on $X$ is $T^{r+s}$ equivariant. If we rescale the
Fubini-Study metric by a number tends to zero, then $X$ will
degenerate to $X_1$, while the $n$ dimensional moment polytope of
$X$ degenerate to a $s$ dimensional moment polytope. This procedure
reproduce the variation above. Thus all fibred toric varieties arise
in this way.

\section{Toric hyperk\"{a}hler variety}
A $4n$-dimensional manifold is hyperk\"{a}hler if it possesses a
Riemannian metric $g$ which is K{\"a}hler with respect to three complex
structures $I_1$; $I_2$; $I_3$ satisfying the quaternionic relations
$I_1 I_2 = -I_2 I_1 = I_3$ etc. To date the most powerful technique
for constructing such manifolds is the hyperk\"{a}hler quotient
method of Hitchin, Karlhede, Lindstrom and Rocek(\cite{HKLR87}). We
specialized on the class of hyperk\"{a}hler quotients of flat
quaternionic space $\mathbb{H}^d$ by subtori of $T^d$. The geometry
of these spaces turns out to be closely connected with the theory of
toric varieties.

\subsection{Symplectic quotient(hyperk\"{a}hler quotient)}

Since $\mathbb{H}^d$ can be identified with $T^*\mathbb{C}^d \cong
\mathbb{C}^d\times\mathbb{C}^d$, it has three complex structures
$\{I_1, I_2, I_3\}$. The real torus $T^d=\{(\zeta_1, \zeta_2,
\cdots, \zeta_d) \in \mathbb{C}^d, |\zeta_i|=1\}$ acts on
$\mathbb{C}^d$ inducing a action on $T^*\mathbb{C}^d$ keeping the
hyperk\"{a}hler structure,
\begin{equation*}
(z,w)\zeta=(z \zeta, w \zeta^{-1}).
\end{equation*}
The subtours $M$ acts on it admitting a hyperk\"{a}hler moment map
$\mu=(\mu_\mathbb{R}, \mu_\mathbb{C}): \mathbb{H}^d \rightarrow
\mathfrak{m}^* \times \mathfrak{m}^*_\mathbb{C}$, given by,
\begin{equation*}
\mu_\mathbb{R}(z,w) =
\frac{1}{2}\sum_{i=1}^d(|z_i|^2-|w_i|^2)\iota^* e^*_i,
\end{equation*}
\begin{equation*}
\mu_\mathbb{C}(z,w) = \sum_{i=1}^d z_i w_i \iota^* e^*_i.
\end{equation*}
where the complex moment map $\mu_\mathbb{C} : \mathbb{H}^d
\rightarrow \mathfrak{m}^*_\mathbb{C}$ is holomorphic with respect
to $I_1$. Bielawski and Dancer introduced the definition of toric
hyperk{\"a}hler varieties, and generally speaking, they are not
toric varieties.

\begin{definition}[\cite{BD00}]
A toric hyperk\"{a}hler variety $Y(\alpha, \beta)$ is a
hyperk\"{a}hler quotient $\mu^{-1}(\alpha, \beta)/M$ for $(\alpha,
\beta) \in \mathfrak{m}^* \times \mathfrak{m}^*_\mathbb{C}$.
\end{definition}

The smooth part of $Y(\alpha,\beta)$ is a $4n$-dimensional
hyperk\"{a}hler manifold, whose hyperk\"{a}hler structure is denoted
by $(g, I_1, I_2, I_3)$. The quotient torus $N=T/M$ acts on
$Y(\alpha,\beta)$, preserving its hyperk\"{a}hler structure. This
residue circle action admits a hyperk\"{a}hler moment map
$\bar{\mu}=(\bar{\mu}_{\mathbb{R}}, \bar{\mu}_{\mathbb{C}})$,
\begin{equation*}
\bar{\mu}_\mathbb{R}([z,w]) =
\frac{1}{2}\sum_{i=1}^d(|z_i|^2-|w_i|^2) e^*_i,
\end{equation*}
\begin{equation*}
\bar{\mu}_\mathbb{C}([z,w]) = \sum_{i=1}^d z_i w_i  e^*_i.
\end{equation*}
Differs from the toric case, the map $\bar{\mu}$ to $\mathfrak{n}^*
\times \mathfrak{n}^*_\mathbb{C}$ is surjective, never with a
bounded image.

Parallel with section 2, we use hyperplane arrangement encoding the
quotient construction. For the moment map takes value in
$\mathfrak{m}^* \times \mathfrak{m}^*_\mathbb{C}$, the lift of
$(\alpha,\beta)$ is $(\lambda^1, \lambda^2, \lambda^3)$, s.t.
\begin{equation*}\label{eq:hyp:lif}
\begin{cases}
&\alpha=\sum \lambda_i^1 \iota^* e_i^*\\
&\beta=\sum (\lambda_i^2+\sqrt{-1}\lambda_i^3 )\iota^* e_i^*
\end{cases}
\end{equation*}
Then we can construct the arrangement of codimension 3 flats (affine
subspaces) in $\mathbb{R}^{3n}$,
\begin{equation*}
H_i=H_i^1 \times H_i^2 \times H_i^3,
\end{equation*}
where
\begin{equation*}\label{eq:hyp:arr}
H_i^h=\{x \in \mathfrak{m}^*|\langle u_i,x \rangle+\lambda_i^h=0\},
\ (h=1,2,3, \ i=1,\dots,d)
\end{equation*}
a prior with orientation $u_i$. For simplicity, we still denote this
arrangement of flats as $\mathcal{A}$. Vice versa, such a
arrangement of $\mathcal{A}$ determines a hyperk{\"a}hler quotient
$Y(\alpha, \beta)$. Different from the toric variety, toric
hyperk{\"a}hler manifolds according to the a arrangement with
different orientations will be biholomorphic to each other(cf.
\cite{vCZ11}).

\begin{example}
Let $\beta=0$ and $\alpha$ corresponds to the arrangement in Example
\ref{ex:1dim:arr}. The resulted toric hyperk\"{a}hler variety is the
desingulariztion of $\mathbb{C}^2/\mathbb{Z}_3$(cf. \cite{HS02},
section 10).
\end{example}

\subsection{Regularity}
Variation the hyperk\"{a}hler structure on a toric hyperk\"{a}hler
variety means altering $(\alpha, \beta)$ in $\mathfrak{m}^* \times
\mathfrak{m}^*_\mathbb{C}$, hence the regularity of $(\alpha,
\beta)$ is a crucial premise.

The hyperk\"{a}hler chamber structure can be defined on
$\mathfrak{m}^* \times \mathfrak{m}^*_\mathbb{C}$ rather than the
positive cone $\mathfrak{m}^*_+$ the same as the toric variety case,
and the walls are entire hyperspaces. For simplicity, we still use
the same notation. There is

\begin{proposition}[\cite{Ko08}]
(1) $(\mathfrak{m}^* \times \mathfrak{m}^*_\mathbb{C})_{reg}=
\mathfrak{m}^* \times \mathfrak{m}^*_\mathbb{C} \backslash
\bigcup^l_{s=1} W_s \otimes W_{s\mathbb{C}}$, where
$W_{s\mathbb{C}}$ is the complexification of $W_s$.

\noindent (2) If $(\alpha, \beta) \in (\mathfrak{m}^* \times
\mathfrak{m}^*_\mathbb{C})_{reg}$, then $X(\alpha, \beta)$ is a
smooth manifold.

\end{proposition}

Similarly, we define the connected components of $(\mathfrak{m}^*
\times \mathfrak{m}^*_\mathbb{C})_{reg}$ to be the chambers. Example
\ref{ex:hyp:chamber} illustrates that although it has the same
expression with the toric one, the underlining structure is
different.

\begin{example}\label{ex:hyp:chamber}
Take $\beta=0$, consider the toric hyperk\"{a}hler varieties
corresponding to the two arrangements in Example
\ref{ex:fanchanged}. We draw their real part chambers in
$\mathfrak{m}^*$ in Figure \ref{fig:hyp:chamber}.
\end{example}

\begin{figure}[h]
\centering \subfigure[the chamber of Example \ref{ex:2dim:arr}]{
\label{fig:hyp:chamber:a} %% label for first subfigure
\includegraphics[width=2.1in]{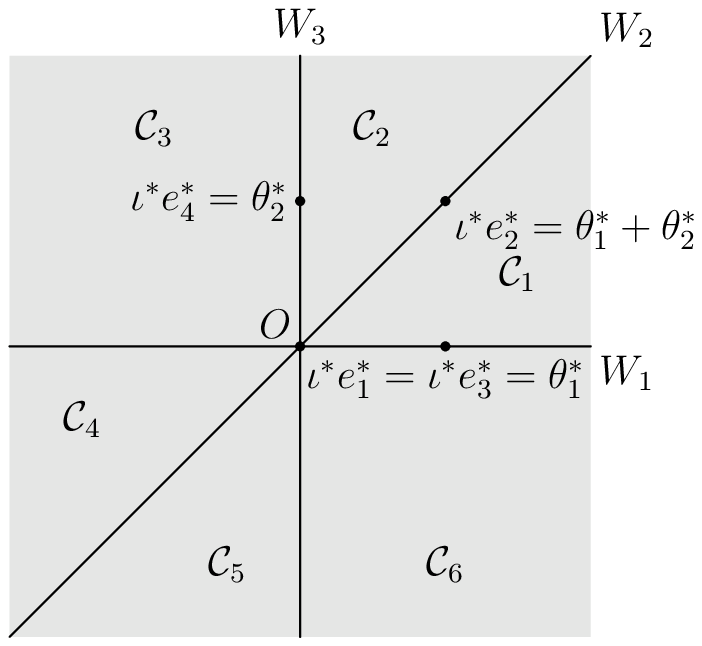}}
\hspace{0.3in} \subfigure[the chamber of same arrangement but
different orientation in Example \ref{ex:fanchanged}]{
\label{fig:hyp:chamber:b} %% label for second subfigure
\includegraphics[width=2.3in]{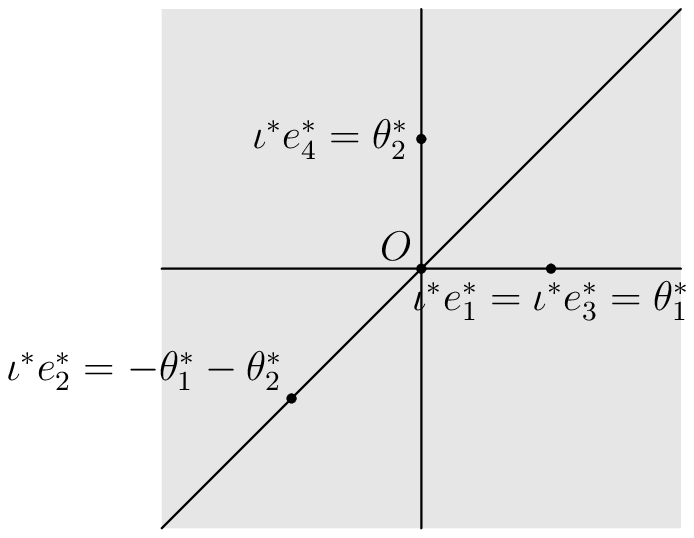}}
\caption{two identical chambers, compare with Figure
\ref{fig:tor:chamber}}
\label{fig:hyp:chamber} %% label for entire figure
\end{figure}

\subsection{Extended core and core}

In this part, $\beta$ is taken to be zero. It is enough to merely
consider the hyperplanes arrangement $H^1$. We will abuse the
notation using $H$ in stead of $H^1$.

The subset
of $Y(\alpha,0)$
\begin{equation*}
Z=\bar{\mu}_\mathbb{C}^{-1}(0)=\{[z,w] \in Y(\alpha,0)| z_i w_i=0 \
\text{for all} \ i\},
\end{equation*}
is called the extended core by Proudfoot(cf. \cite{Pr04}), which
naturally breaks into components
\begin{equation*}
Z_\epsilon=\{[z,w] \in Y(\alpha,0)| w_i=0 \ if \ \epsilon(i)=1 \ and
\ z_i=0 \ if \ \epsilon(i)=-1\}.
\end{equation*}
The variety $Z_\epsilon \subset Y(\alpha,0)$ is a $n$-dimensional
isotropy K{\"a}hler subvariety of $Y(\alpha,0)$ with an effective
hamiltonian $T^n$-action, hence a toric variety itself. It is not
hard to see this is just the toric variety corresponding to the
oriented hyperplane arrangement $\mathcal{A}_\epsilon$. Denote the
associated polytope of $\mathcal{A}_\epsilon$ as $\Delta_\epsilon$.
The set $Z_{cpt}=\bigcup_{\epsilon \in \Theta_{cpt}} Z_\epsilon$,
where $\Theta_{cpt}=\{\epsilon|\Delta_\epsilon \ \text{bounded}\}$,
is called the core, union of compact toric varieties
$X(\mathcal{A}_\epsilon), \epsilon \in \Theta_{cpt}$.

\section{Variation of toric hyperk\"{a}hler variety}

\subsection{GIT quotient}
It's time to establish the GIT quotient description of toric
hyperk\"{a}hler variety.
%\subsubsection{Definition}
Consider the GIT quotient of the affine variety
$\mu_{\mathbb{C}}^{-1}(\beta)$ by $M_\mathbb{C}$ with respect to the
linearization on the trivial line bundle induced by $\alpha \in
\mathfrak{m}^*_\mathbb{Z}$. Denote the set of $\alpha$-semi-stable
points in $\mu_{\mathbb{C}}^{-1}(\beta)$ by
$\mu_{\mathbb{C}}^{-1}(\beta)^{\alpha-ss}$, then there is a
categorical quotient $\phi: \mu_{\mathbb{C}}^{-1}(\beta)^{\alpha-ss}
\rightarrow
\mu_{\mathbb{C}}^{-1}(\beta)^{\alpha-ss}//M_{\mathbb{C}}$. Parallel
with toric case, the stability condition can be generalized to any
$\alpha \in \mathfrak{m}^*$(cf. \cite{Ko08}).

\begin{definition}
Suppose that $\alpha \in \mathfrak{m}^*$,

\noindent (1) A point $(z,w) \in \mu_{\mathbb{C}}^{-1}(\beta)$ is
$\alpha$-semi-stable if and only if
\begin{equation}\label{eq:hyp:sta}
\alpha \in \sum^d_{i=1} \mathbb{R}_{\geq 0}|z_i|^2 \iota^* e_i^* +
\sum^d_{i=1}\mathbb{R}_{\geq 0}|w_i|^2 (-\iota^* e_i^*).
\end{equation}

\noindent (2) Suppose $(z,w) \in
\mu_{\mathbb{C}}^{-1}(\beta)^{\alpha-ss}$. Then the
$M_\mathbb{C}$-orbit through $(z,w)$ is closed in
$\mu_{\mathbb{C}}^{-1}(\beta)^{\alpha-ss}$ if and only if
\begin{equation}\label{eq:hyp:clo}
\alpha \in \sum^d_{i=1} \mathbb{R}_{> 0}|z_i|^2 \iota^* e_i^* +
\sum^d_{i=1}\mathbb{R}_{> 0}|w_i|^2 (-\iota^* e_i^*).
\end{equation}
\end{definition}

And then

\begin{lemma}
For any point $p \in
\mu_{\mathbb{C}}^{-1}(\beta)^{\alpha-ss}//M_{\mathbb{C}}$, the fiber
$\phi^{-1}(p)$ consists of finitely many $M_{\mathbb{C}}$-orbits.
Moreover, each fiber contains the unique closed
$M_{\mathbb{C}}$-orbits in
$\mu_{\mathbb{C}}^{-1}(\beta)^{\alpha-ss}$. Thus the categorical
quotient $\mu_{\mathbb{C}}^{-1}(\beta)^{\alpha-ss}//M_{\mathbb{C}}$
can be identified with a set of closed $M_{\mathbb{C}}$-orbits in
$\mu_{\mathbb{C}}^{-1}(\beta)^{\alpha-ss}$.
\end{lemma}

Thus we can identify the symplectic quotient $Y(\alpha, \beta)$ with
the GIT quotient
$\mu_{\mathbb{C}}^{-1}(\beta)^{\alpha-ss}//M_{\mathbb{C}}$ for any
$\alpha \in \mathfrak{m}^*$.

In the remain part of this subsection, we recall results about
variation in \cite{Ko08}, and reinterpret them from the perspective
of fibred toric variety. We discuss the natural morphism in detail,
and the Mukai flop can be viewed as the adjunction of two adjoint
natural morphisms.

\subsection{Natural morphism and Mukai flop}

Under the same set up of toric case, consider the real part chamber
structure for a fixed $\beta$, $\alpha_1$ is in a generic position
of the wall $W_1$ and $\alpha^+$, $\alpha^-$ lie in the chamber
$\mathcal{C}^+$ and $\mathcal{C}^-$ beside the wall. By the
definition of stability, we have
\begin{equation*}
\mu_{\mathbb{C}}^{-1}(\beta)^{\alpha^\pm-ss}
 \subset \mu_{\mathbb{C}}^{-1}(\beta)^{\alpha_1-ss}.
\end{equation*}
Thus this inclusion induces a natural morphism from GIT quotients
$\mu_{\mathbb{C}}^{-1}(\beta)^{\alpha^\pm-ss}//M_\mathbb{C}$ to
another GIT quotient
$\mu_{\mathbb{C}}^{-1}(\beta)^{\alpha_1-ss}//M_\mathbb{C}$, which we
denote by $ \pi^\pm: (Y(\alpha^\pm, \beta), I_1) \rightarrow
(Y(\alpha_1, \beta), I_1)$. Without ambiguity we still utilize the
notation in toric case. Konno proved Theorem \ref{thm:hyp:nat}, and for the reader's convenience, we give the sketch of proof.

\noindent \textbf{The Proof of Theorem \ref{thm:hyp:nat}:} (1)
Similar with the toric case, $(\alpha_1, \beta)$ can be considered
as a regular element of $\mathfrak{m}_1 \times
\mathfrak{m}_{1\mathbb{C}}$. Likewise $V_1$ is a hyperk\"{a}hler
quotient of $\mathbb{H}^{\{1,...,d\}\backslash J_1}= \{(z,w) \in
\mathbb{H}^d | z_i = w_i = 0 \ if \ i \in J_1\}$ by $M_1$.

\noindent (2) Choosing $\theta_1 \in \mathrm{Lie} G_1$, by Equation
(\ref{eq:hyp:sta}) we can show that
$\mu_{\mathbb{C}}^{-1}(\beta)^{\alpha^+-ss}$ is exactly the points
in $\mu_{\mathbb{C}}^{-1}(\beta)^{\alpha_1-ss}$ satisfying
\begin{equation*}
\text{there exists} \ i \in J_1 \ \text{such that} \ z_i \neq 0 \
\text{if} \ i \in J^+_1 \ \text{or} \ w_i \neq 0 \ \text{if} \ i \in
J^-_1.
\end{equation*}
It is also easy to see that, if $(z,w) \in
\mu_{\mathbb{C}}^{-1}(\beta)^{\alpha^+-ss}$, then
\begin{equation*}
[z,w] \in V_1 \  \text{is equivalent to} \ w_i = 0 \ \text{for} \ i
\in J^+_1 \ \text{and} \ z_i = 0 \ \text{for} \ i \in J^-_1.
\end{equation*}
Thus the fiber of $\pi^+|_{V^+} : V^+ \rightarrow V_1$ is
biholomorphic to $(\mathbb{C}^{\#J_1} \backslash
\{0\})/G_{1\mathbb{C}}$, i.e. $\mathbb{C}P^{\#J_1-1}$. Same thing
happens for $\alpha^-$.

\noindent (3) Same with toric case, reader could consult
\cite{Ko08}. \qed

Konno also studied the cross wall phenomena, and show that it turns
out to be the Mukai's elementary transform.

\begin{theorem}
Assume $\alpha^+ \in \mathcal{C}^+$ and $\alpha^- \in \mathcal{C}^-$
at different sides of wall $W_1$, we can relate $(Y(\alpha^+,
\beta), I_1)$ to $(Y(\alpha^-, \beta), I_1)$ by a Mukai flop.
Especially, If $\# J_1 = 2$, there exists a biholomorphic map
$\varphi : (Y(\alpha^+, \beta), I_1) \rightarrow (Y(\alpha^-,
\beta), I_1)$ satisfying $\pi^+= \pi^- \circ \varphi$.

\end{theorem}

There is no boundary wall and in the interior wall case, toric
hyperk\"{a}hler version is much simpler than the toric one, for all the $J_1$ will contribute to the fiber rather than only the
$J_1^+$. Now, we are ready to prove Theorem \ref{thm:restri}, and
$\beta$ is set to be zero.

\subsection{Relation with fibred toric variety}

Take $\beta=0$, consider the natural morphism between
$4n$-dimensional toric hyper{\"a}hler varieties $\pi^\pm:
Y(\alpha^\pm,0) \rightarrow Y(\alpha_1,0)$, it can be encoded as the
variation of a regular hyperplane arrangement $\mathcal{A}$ to a
non-simplicial one $\mathcal{A}_1$. There are some ploytopes belong
to the hyperplane arrangement $\mathcal{A}$ vanish in this
procedure. These ploytopes are fibred polytopes corresponding to
fibred toric varieties. More precisely, the singular set of
$\mathcal{A}$ is a $n-(\# J_1-1)$ dimensional arrangement
$\mathcal{S}$ constituted by $d-\# J_1$ hyperplanes, then $V_1$ is
the toric hyperk{\"a}hler variety defined by $\mathcal{S}$. Denoting
$\tilde{\epsilon}$ a map form $\{1,\dots, d-\# J_1\}$ to $\{-1,1\}$,
we can state our main theorem as

\begin{theorem*}

Let $Z_1$ be the extended core of $V_1$ which is the toric varieties $X(\mathcal{S}_{\tilde{\epsilon}})$ intersecting together,
Restrict $V^\pm$ these $\mathbb{C}P^r$ fiber bundles to $Z_1$, then $V^\pm|_{X(\mathcal{S}_{\tilde{\epsilon}})}$ the $\mathbb{C}P^r$ fiber bundles over each $X(\mathcal{S}_{\tilde{\epsilon}})$ are all fibred toric varieties of complex dimension $n$.

\end{theorem*}

\noindent \textbf{The proof of Theorem \ref{thm:restri}:} Here we
only consider $V^+$, the $V^-$ case is the same. We first show the
bundle over $Z_1$ lies in the extended core $Z^+$ of $Y(\alpha^+)$.
This is merely a repeat of Konno's proof for (2) of Theorem
\ref{thm:hyp:nat}. We had already know that $[z,w] \in V^+$ is
equivalent to $w_i = 0$ for $i \in J^+_1$ and $z_i = 0$ for $i \in
J^-_1$, and the extended core of $V_1$ is $\{[z,w]\in
\mathbb{H}^{\{1,...,d\}\backslash J_1}|z_i w_i=0, \text{for} \ i\in
\{1,...,d\}\backslash J_1\}$. Combining these two facts, we know
$z_iw_i=0$ for all $i$, thus the restricted bundle lies in $Z$.
Secondly, By Theorem \ref{thm:hyp:nat}, $V^+$ is complex dimension
$2n-2(\# J_1-1)$. Hence $X(\mathcal{S}_{\tilde{\epsilon}})$ has half
dimension of $V_1$, $n-(\# J_1-1)$. While its fiber is
$(\#J_1-1)$-dimensional complex projective space, thus the dimension
counting tells us it is complex $n$-dimensional. Finally, the volume
of these toric varieties vanishing during the variation, thus must
be the fibred toric varieties associated to the fibred polytopes.
\qed

This is equivalent to say that the hyperk\"{a}hler natural morphisms
restricted to the toric varieties in the extended core, are the
natural morphisms of respect toric varieties. Especially to the
fibred toric varieties, the natural morphisms degenerate to bundle
projections. As we know, besides the bundle projection of fibred
toric varieties, there are lots of "flip" of toric varieties.
Theorem \ref{thm:restri} tells us that if we look at these flips in
its ambient toric hyperk\"{a}hler variety, then they are all
contained in nearby bundle projections of fibred toric varieties.
Thus the fibred toric varieties are both primary in the variation of
toric varieties and toric hyperk\"{a}hler varieties. Moreover,
Theorem \ref{thm:restri} directly implies

\begin{corollary}
Every smooth toric hyperk{\"a}hler manifold contain fibred toric
varieties in its extended core.
\end{corollary}

We close this article by several examples.

\begin{example}\label{ex:desingular}
We have diagonal action of $T^1$ on $\mathbb{H}^2$, trying to derive
the natural morphism from $Y(1,0)$ to $Y(0,0)$. We already know
$Y(1,0)=T^*\mathbb{C}P^1$. While in the case $\alpha=0$, for the
whole affine set $\{z_1 w_1+z_2 w_2=0\}$ is semi-stable, the GIT
quotient is just affine quotient. The invariant polynomial ring
$K^T[z_1,z_2,w_1,w_2]/\{z_1w_1+z_2w_2\}$ is of form $z_1^p z_2^q
w_1^r w_2^s$ where $p+q=r+s$, which can be generated by $X=z_1w_1$,
$Y=z_1w_2$, $Z=z_2w_1$ and $W=z_2w_2$. Thus the ring is isomorphic
to $K[X,Y,Z,W]/\{XW-YZ, X+W\}$, i.e. $K[X,Y,Z]/\{X^2+YZ\}$. Its
corresponding affine variety is the affine cone $X^2+YZ=0$ in
$\mathbb{C}^3$, isomorphic to $\mathbb{C}^2/\mathbb{Z}_2$(cf.
section 2.2 of \cite{Fu93}). It means that the natural morphism from
is just the Klein desigularization of $\mathbb{C}^2/\mathbb{Z}_2$ to
$T^*\mathbb{C}P^1$.

\end{example}

It is easy to check, the extended core of $Y(1,0)$ is the union of two copies of $\mathbb{C}$ and
$\mathbb{C}P^1$.

\begin{example}
Taking $\beta=0$, we consider the fan in Example \ref{ex:2dim:arr},
and take $\alpha=3 \theta_1^*+2\theta_2^* \in \mathcal{C}_1$ and
$\alpha_1=3 \theta_1^* \in W_1$, which corresponds to the lifts
$\lambda=(1,1,1,1)$ and $\lambda_1=(1,1,1,-1)$, then the variation
from $\alpha$ to $\alpha_1$ is equivalent to moving $H_4$ to the
superposition of $H_2$(see Figure \ref{fig:hyp:var}). In natural
morphism $\pi: Y(\alpha,0) \rightarrow Y(\alpha_1,0)$, the singular
set $V_1$ in $Y(\alpha_1,0)$ is $T^*\mathbb{C}P^1$, which
corresponds to the induced arrangement on the singular set
$H_2=H_4$, with extended core the union of two copies of
$\mathbb{C}$ and $\mathbb{C}P^1$. The exceptional set $V$ in
$Y(\alpha)$ is a $\mathbb{C}P^1$ bundle over $T^*\mathbb{C}P^1$.
Restrict it to the extended core, we get two trivial $\mathbb{C}P^1$
bundles over $\mathbb{C}$ and a Hirzebrunch surface, which are the
fibred toric varieties.

\end{example}

\begin{figure}[h]
\centering \subfigure[the fibred toric varieties associated]{
\label{fig:hyp:var:reg} %% label for first subfigure
\includegraphics[width=2.2in]{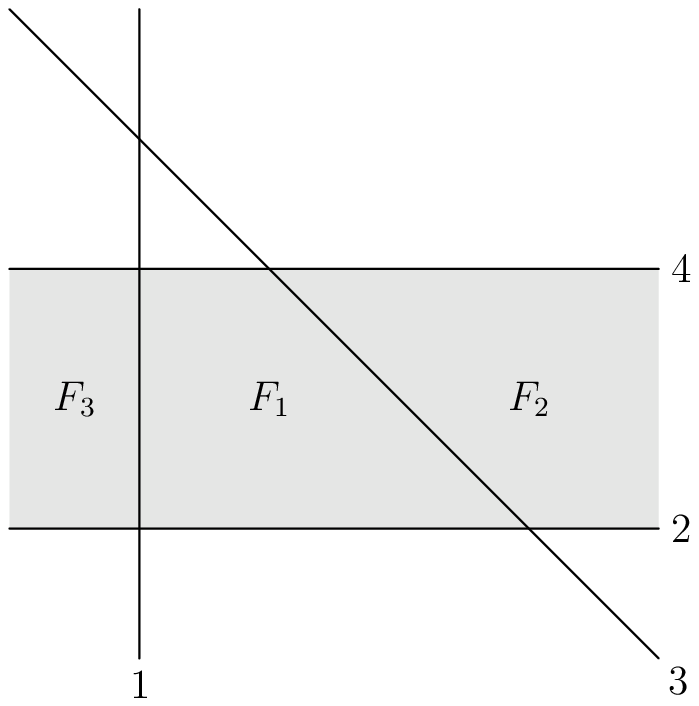}}
\hspace{0.3in} \subfigure[the non-simplicial arrangement, where ploytopes
$F_1$, $F_2$ and $F_3$ vanish]{
\label{fig:hyp:var:sin} %% label for second subfigure
\includegraphics[width=2.2in]{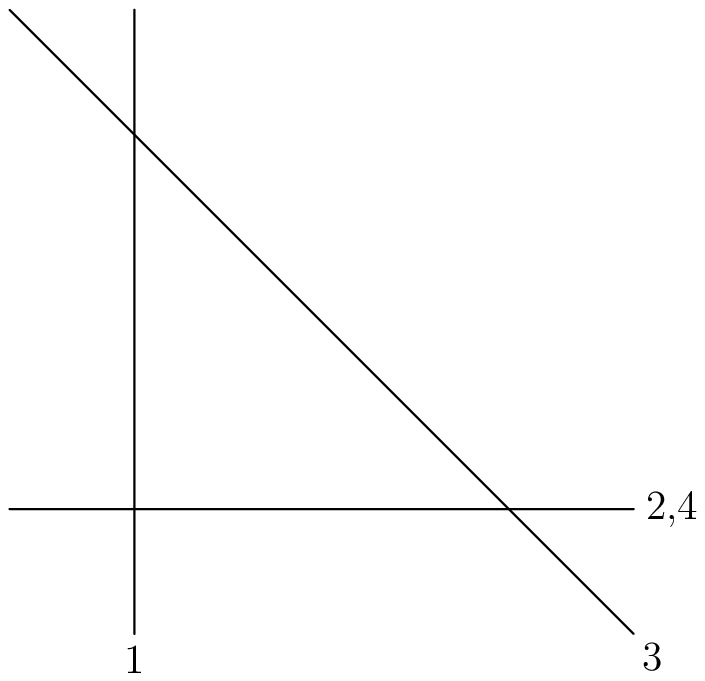}}
\caption{variation of arrangement}
\label{fig:hyp:var} %% label for entire figure
\end{figure}

\bibliographystyle{alpha}
\bibliography{hkBib}

\vfill

\noindent Craig van Coevering, Email address: craigvan@ustc.edu.cn

\

\noindent Wei Zhang, Email address: zhangw81@ustc.edu.cn

\

\noindent School of Mathematical Sciences

\noindent University of Science and Technology of China

\noindent Hefei, 230026, P. R.China

\noindent

\end{document}